\documentclass[a4paper]{amsart}

\usepackage{amsfonts, latexsym, amssymb, graphicx,
multicol, mathrsfs, color, amscd, verbatim, paralist,
xspace, url, euscript, stmaryrd,  amsmath, amscd, enumitem,
bbold, multirow, tikz, mathtools, mathabx}
\usepackage[all,pdf,cmtip]{xy}

\usepackage[colorlinks, linkcolor=blue, citecolor=magenta, urlcolor=cyan]{hyperref}

\usepackage{mathabx}

\def\ra{\rightarrow}

\newtheorem{theorem}{THEOREM}[section]
\newtheorem{corollary}[theorem]{Corollary}
\newtheorem{proposition}[theorem]{Proposition}

\newtheorem{example}[theorem]{Example}
\theoremstyle{definition}

\theoremstyle{remark}
\newtheorem{remark}[theorem]{Remark}

\newcommand\CC{{\mathbb C}}
\newcommand\RR{{\mathbb R}}

\def\GL{\mathop{\rm GL}\nolimits}

\def\Re{\mathop{\rm Re}\nolimits}

\makeatletter
\def\blfootnote{\xdef\@thefnmark{}\@footnotetext}
\makeatother

\begin{document}

\title[On the CR-curvature of tube hypersurfaces]{On the CR-curvature of Levi degenerate
\vspace{0.1cm}\\
tube hypersurfaces}\blfootnote{{\bf Mathematics Subject Classification:} 32V05, 32V20, 35J96, 34A05, 34A26.} \blfootnote{{\bf Keywords:} CR-curvature, the Monge-Amp\`ere equation, the Monge equation.}
\author[Isaev]{Alexander Isaev}

\address{Mathematical Sciences Institute\\
Australian National University\\
Acton, ACT 2601, Australia}
\email{alexander.isaev@anu.edu.au}

\maketitle

\thispagestyle{empty}

\pagestyle{myheadings}

\begin{abstract} 
In article \cite{I2} we studied tube hypersurfaces in $\CC^3$ that are 2-nondegenerate and uniformly Levi degenerate of rank 1. In particular, we discovered that for the CR-curvature of such a hypersurface to vanish it suffices to require that only two coefficients (called $\Theta^2_{21}$ and $\Theta^2_{10}$) in the expansion of a certain component of the CR-curvature form be identically zero. In this paper, we show that, surprisingly, the vanishing of the entire CR-curvature is in fact implied by the vanishing of a single quantity derived from $\Theta^2_{10}$. This result strengthens the main theorem of \cite{I2} and also leads to a remarkable system of partial differential equations. Furthermore, we explicitly characterize the class of not necessarily CR-flat tube hypersurfaces given by the vanishing of $\Theta^2_{21}$.
\end{abstract}

\section{Introduction}\label{intro}
\setcounter{equation}{0}

This paper extends our earlier article \cite{I2}, and the reader will be extensively referred to \cite{I2} in what follows. In particular, a detailed review of all the  necessary CR-geometric concepts is contained in \cite[Section 2]{I2}, and we will utilize those concepts without further reference.

We consider connected $C^{\infty}$-smooth real hypersurfaces in the complex vector space $\CC^n$ with $n\ge 2$. Specifically, we study {\it tube hypersurfaces}, or simply {\it tubes}, i.e, locally closed real submanifolds of the form
$$
M={\mathcal S}+iV,
$$
where ${\mathcal S}$ is a hypersurface in $\RR^n\subset\CC^n$ called the base of $M$. Two tube hypersurfaces are called affinely equivalent if there exists an affine transformation of $\CC^n$ given by
\begin{equation}
z\mapsto Az+b,\quad A\in\GL_n(\RR),\quad b\in\CC^n\label{affequiv}
\end{equation}
that maps one hypersurface onto the other (this occurs if and only if the bases of the tubes are affinely equivalent as submanifolds of $\RR^n$).

There has been a substantial effort to relate the CR-geometric and affine-geometric aspects of the study of tubes (see \cite[Section 1]{I2} for an extensive bibliography). Specifically, the following question has attracted much attention:
\vspace{-0.5cm}\\

$$
\begin{array}{l}
\hspace{0.2cm}\hbox{$(*)$ when does local or global CR-equivalence of tubes imply}\\
\vspace{-0.3cm}\hspace{0.8cm}\hbox{affine equivalence?}\\
\end{array}
$$
\vspace{-0.4cm}\\

\noindent Until recently, a reasonable answer to the above question has only existed for Levi nondegenerate tube hypersurfaces that are also CR-flat, i.e., have identically vanishing CR-curvature (see monograph \cite{I1} for an up-to-date exposition of the existing theory). In an attempt to relax the Levi nondegeneracy requirement, in \cite{I2} we set out to investigate question $(*)$ for a class of Levi degenerate 2-nondegenerate hypersurfaces while still assuming CR-flatness. As part of our considerations, we analyzed CR-curvature for this class, and in the present paper we improve on that analysis. 

Notice that CR-curvature is only defined in situations when the CR-structures in question are reducible to absolute parallelisms with values in some Lie algebra ${\mathfrak g}$. Indeed, let ${\mathfrak C}$ be a class of CR-manifolds. Then the CR-structures in ${\mathfrak C}$ are said to reduce to ${\mathfrak g}$-valued absolute parallelisms if to every $M\in{\mathfrak C}$ one can assign a fiber bundle ${\mathcal P}_M\ra M$ and an absolute parallelism $\omega_M$ on ${\mathcal P}_M$ such that for every $p\in M$ the parallelism establishes an isomorphism between $T_p(M)$ and ${\mathfrak g}$ and for any $M_1,M_2\in{\mathfrak C}$ the following holds: 

\noindent (i) every CR-isomorphism $f:M_1\ra M_2$ can be lifted to a diffeomorphism\linebreak $F: {\mathcal P}_{M_{{}_1}}\ra{\mathcal P}_{M_{{}_2}}$ satisfying
\begin{equation}
F^{*}\omega_{M_{{}_2}}=\omega_{M_{{}_1}},\label{eq8}
\end{equation}
and 

\noindent (ii) any diffeomorphism $F: {\mathcal P}_{M_{{}_1}}\ra{\mathcal P}_{M_{{}_2}}$ satisfying (\ref{eq8}) 
is a bundle isomorphism that is a lift of a CR-isomorphism $f:M_1\ra M_2$. 

In this situation one introduces the ${\mathfrak g}$-valued {\it CR-curvature form}\,
$$
\Omega_M:=d\omega_M-\frac{1}{2}\left[\omega_M,\omega_M\right],\label{genformulacurvature}
$$
and the condition of the CR-flatness of $M$ means that $\Omega_M$ identically vanishes on the bundle ${\mathcal P}_M$.
 
Reducing CR-structures (as well as other geometric structures) to absolute parallelisms goes back to \'E. Cartan  who showed that reduction takes place for all 3-dimensional Levi nondegenerate CR-hyper\-surfaces (see \cite{C}). Since then there have been many developments incorporating the assumption of Levi nondegeneracy (see \cite[Section 1]{I2} for references). On the other hand, reducing the CR-structures of Levi degenerate CR-mani\-folds has proved to be rather difficult, and the first result for a large class of Levi degenerate manifolds only appeared in 2013 in our paper \cite{IZ}. Specifically, we considered the class ${\mathfrak C}_{2,1}$ of connected 5-dimensional CR-hypersurfaces that are 2-nondegenerate and uniformly Levi degenerate of rank 1 and showed that the CR-structures in this class reduce to ${\mathfrak{so}}(3,2)$-valued parallelisms (see \cite{MS}, \cite{Poc} for alternative constructions and \cite{Por} for a reduction in the 7-dimensional case). In particular, in \cite{IZ} we prove that a manifold $M\in{\mathfrak C}_{2,1}$ is CR-flat (with respect to our reduction) if and only if near its every point $M$ is CR-equivalent to an open subset of the tube hypersurface over the future light cone in $\RR^3$:
$$
M_0:=\left\{(z_1,z_2,z_3)\in\CC^3\mid (\Re z_1)^2+(\Re z_2)^2-(\Re z_3)^2=0,\,\ \Re z_3>0\right\}.\label{light}
$$

Now, the main result of \cite{I2} (see Theorem 1.1 therein) asserts that every CR-flat tube hypersurface in ${\mathfrak C}_{2,1}$ is affinely equivalent to an open subset of $M_0$. This conclusion is a complete answer to question $(*)$ in the situation at hand and is in stark contrast to the Levi nondegenerate case where the CR-geometric and affine-geometric classifications differ even in low dimensions. 

In fact, in \cite{I2} we obtain a stronger result. Namely, for the assertion of \cite[Theorem 1.1]{I2} to hold, it suffices to require that only two coefficients (called $\Theta^2_{21}$ and $\Theta^2_{10}$) in the expansion of a single component of the CR-curvature form $\Omega_M$  (called $\Theta^2$) be identically zero on ${\mathcal P}_M$ (see \cite[Theorem 3.1]{I2}). Our argument is local, and for every $x\in M$ we only utilize the vanishing of $\Theta^2_{21}$ and $\Theta^2_{10}$ on a particular section $\gamma$ of ${\mathcal P}_M$ over a neighborhood of $x$:
\begin{equation}
\left\{\begin{array}{l}
\Theta^2_{21}|_{\gamma}=0,\\
\vspace{-0.1cm}\\
\Theta^2_{10}|_{\gamma}=0.
\end{array}\right.\label{ceqs1}
\end{equation} 
Each of the two conditions in system (\ref{ceqs1}) can be expressed as a partial differential equation on the local defining function of the hypersurface $M$ near $x$. These equations are quite complicated; for example, the first identity in (\ref{ceqs1}) is equivalent to (\ref{veryfinalthetav}). The expression for $\Theta^2_{10}|_{\gamma}$ is especially hard to find, and in our computation of $\Theta^2_{10}|_{\gamma}$ in \cite{I2} some of its terms were only calculated under the simplifying assumption $\Theta^2_{21}|_{\gamma}=0$. This was sufficient for our purposes as we were only interested in solving system (\ref{ceqs1}). Indeed, denoting by ${\mathbf \Theta}^2_{10}$ the quantity arising from the constrained calculation of $\Theta^2_{10}|_{\gamma}$, we see that the system of equations  
\begin{equation}
\left\{\begin{array}{l}
\Theta^2_{21}|_{\gamma}=0,\\
\vspace{-0.1cm}\\
{\mathbf \Theta}^2_{10}=0
\end{array}\right.\label{ceqs}
\end{equation}
is equivalent to (\ref{ceqs1}). Interestingly, if in suitable coordinates the base of $M$ is given locally as the graph of a function of two variables, the second equation in (\ref{ceqs}) becomes the well-known Monge equation on this function with respect to one of the variables (see (\ref{veryfinalthetasss})). 

To write system (\ref{ceqs}) more explicitly, recall that $M$ is uniformly Levi degenerate of rank 1 and 2-nondegenerate. Due to Levi degeneracy, the graphing function of $M$ satisfies the homogeneous Monge-Amp\`ere equation (see (\ref{mongeampere})). Thus, the detailed form of (\ref{ceqs}) is
\begin{equation}
\left\{\begin{array}{l}
\Theta^2_{21}|_{\gamma}=0,\\
\vspace{-0.1cm}\\
\hbox{The Monge equation w.r.t. one variable:}\,\,{\mathbf \Theta}^2_{10}=0,\\
\vspace{-0.1cm}\\
\hbox{The Monge-Amp\`ere equation},
\end{array}\right.\label{threeeqs}
\end{equation}
where we additionally assume that certain quantities responsible for the Levi form to have rank precisely 1 and for 2-nondegen\-eracy are everywhere nonzero (see (\ref{rho11}) and (\ref{snonzero}), respectively). System (\ref{threeeqs}) is the centerpiece of the proof of \cite[Theorem 3.1]{I2}. In \cite{I2} we explicitly solved (\ref{threeeqs}) and observed that every solution of this system defines a tube hypersurface affinely equivalent to an open subset of $M_0$. As $M_0$ is CR-flat, this shows, in particular, that conditions (\ref{ceqs}) imply the vanishing of the CR-curvature form $\Omega_M$ on an open subset of the bundle ${\mathcal P}_M$ over a neighborhood of $x$. Hence if both $\Theta^2_{21}$ and $\Theta^2_{10}$ are identically zero on ${\mathcal P}_M$, so is the entire form $\Omega_M$.

The main theorem of the present paper establishes a surprising dependence between the two local conditions in (\ref{ceqs}). We will now state the theorem in general terms, with the detailed formulation postponed until the next section (see\linebreak Theorem \ref{maindetailed}).

\begin{theorem}\label{main}
Let $M$ be a tube hypersurface in $\CC^3$ and assume that $M\in{\mathfrak C}_{2,1}$. Fix $x\in M$ and a suitable section $\gamma$ of ${\mathcal P}_M$ over a neighborhood of $x$. Then the condition ${\mathbf \Theta}^2_{10}=0$ implies $\Theta^2_{21}|_{\gamma}=0$. 
\end{theorem}

\noindent We stress that although the quantity ${\mathbf \Theta}^2_{10}$ was computed in part under the assumption $\Theta^2_{21}|_{\gamma}=0$, it is not at all clear {\it a priori}\, why the vanishing of ${\mathbf \Theta}^2_{10}$ should imply that of $\Theta^2_{21}|_{\gamma}$.

Together with results of \cite{I2}, Theorem \ref{main} yields:

\begin{corollary}\label{flatness}
Let $M$ be a tube hypersurface in $\CC^3$ and assume that $M\in{\mathfrak C}_{2,1}$. Fix $x\in M$ and a suitable section $\gamma$ of ${\mathcal P}_M$ over a neighborhood of $x$. If ${\mathbf \Theta}^2_{10}=0$, then $M$ near $x$ is affinely equivalent to an open subset of $M_0$; in particular, the CR-curvature form $\Omega_M$ vanishes on an open subset of ${\mathcal P}_M$ over a neighborhood\linebreak of $x$.
\end{corollary}

\noindent The above result is rather unexpected as it has been believed for some time now that CR-flatness for manifolds in the class ${\mathfrak C}_{2,1}$ should be controlled by two conditions rather than one (cf.~Remark \ref{pocfunctions}).

By Theorem \ref{main}, system (\ref{threeeqs}) reduces to a system of two equations:
\begin{equation}
\left\{\begin{array}{l}
\hbox{The Monge equation w.r.t. one variable:}\,\,{\mathbf \Theta}^2_{10}=0,\\
\vspace{-0.1cm}\\
\hbox{The Monge-Amp\`ere equation},
\end{array}\right.\label{twoeeqs}
\end{equation}
where we assume in addition that (\ref{rho11}) and (\ref{snonzero}) are satisfied. This system is truly remarkable. Indeed, by Corollary \ref{flatness} it has a clear geometric meaning as it locally describes all CR-flat tubes in the class ${\mathfrak C}_{2,1}$. Moreover, all solutions of this system can be explicitly found, and every solution yields a tube hypersurface affinely equivalent to an open subset of $M_0$. Next, each of the two equations in (\ref{twoeeqs}) has its own geometric interpretation: the classical single-variable Monge equation describes all planar conics (see, e.g., \cite[pp.~51--52]{Lan}, \cite{Las}), whereas the graphs of the solutions of the Monge-Amp\`ere equation are exactly the surfaces in $\RR^3$ with degenerate second fundamental form. Finally---and quite curiously---both equations in (\ref{twoeeqs}) happen to be named after Gaspard Monge. It is rather satisfying to see that the invariants constructed in \cite{IZ} lead to an object so abundantly filled with geometric features. This indicates that the theory of the class ${\mathfrak C}_{2,1}$ is rich and deserves further exploration.

The paper is organized as follows. In Section \ref{secmain} we state and prove Theorem \ref{maindetailed}, which is the detailed variant of Theorem \ref{main}. Further, in Section \ref{secother} we investigate the converse implication, namely the question whether the vanishing of $\Theta^2_{21}|_{\gamma}$ implies that of ${\mathbf \Theta}^2_{10}$. The answer to this question turns out to be negative, and in Propositions \ref{main1}, \ref{firstcondrho} we write the general form of a solution of the system
\begin{equation}
\left\{\begin{array}{l}
\Theta^2_{21}|_{\gamma}=0,\\
\vspace{-0.1cm}\\
\hbox{The Monge-Amp\`ere equation,}
\end{array}\right.\label{twoeeqs1}
\end{equation}
where, as before, we assume that (\ref{rho11}) and (\ref{snonzero}) are satisfied. Unlike (\ref{twoeeqs}), system (\ref{twoeeqs1}) describes a class of not necessarily CR-flat tubes in ${\mathfrak C}_{2,1}$, and Propositions \ref{main1}, \ref{firstcondrho} show that this interesting class can be effectively characterized as well.    

{\bf Acknowledgements.} This work is supported by the Australian Research Council. The author is grateful to Boris Kruglikov for useful discussions.

\section{The main result}\label{secmain}
\setcounter{equation}{0}

Let $M$ be any tube hypersurface in $\CC^3$. For $x\in M$, a tube neighborhood of $x$ in $M$ is an open subset $U$ of $M$ that contains $x$ and has the form $M\cap({\mathcal U}+i\RR^3)$, where ${\mathcal U}$ is an open subset of $\RR^3$. It is easy to see that for every point $x\in M$ there exists a tube neighborhood $U$ of $x$ in $M$ and an affine transformation of $\CC^3$ as in (\ref{affequiv}) that maps $x$ to the origin and establishes affine equivalence between $U$ and a tube hypersurface of the form
\begin{equation}
\begin{array}{l}
\Gamma_{\rho}:=\{(z_1,z_2,z_3): z_3+{\bar z}_3=\rho(z_1+{\bar z}_1,z_2+{\bar z}_2)\}=\\
\vspace{-0.1cm}\\
\hspace{4cm}\displaystyle\left\{(z_1,z_2,z_3): \Re z_3=\frac{1}{2}\,\rho(2\Re z_1,2\Re z_2)\right\},
\end{array}\label{basiceq}
\end{equation}
where $\rho(t_1,t_2)$ is a smooth function defined in a neighborhood of 0 in $\RR^2$ with
\begin{equation}
\rho(0)=0,\quad \rho_1(0)=0, \quad \rho_2(0)=0\label{initial}
\end{equation}
(here and below subscripts 1 and 2 indicate partial derivatives with respect to $t_1$ and $t_2$). In what follows, $\Gamma_{\rho}$ will be analyzed locally near the origin, thus we will only be interested in the germ of $\rho$ at 0 and the domain of $\rho$ will be allowed to shrink if necessary.

Let now $M$ be uniformly Levi degenerate of rank 1. Then the Hessian matrix of $\rho$ has rank 1 at every point, hence $\rho$ is a solution of the homogeneous Monge-Amp\`ere equation
\begin{equation}
\rho_{11}\rho_{22}-\rho_{12}^2=0,\label{mongeampere}
\end{equation}
where one can additionally assume
\begin{equation}
\hbox{$\rho_{11}>0$ everywhere.}\label{rho11}
\end{equation}
In \cite[Section 3]{I2} we showed that for $\rho$ satisfying (\ref{mongeampere}), (\ref{rho11}), the hypersurface $\Gamma_{\rho}$ is 2-nondegenerate if and only if the function
\begin{equation}
S:=\left(\frac{\rho_{12}}{\rho_{11}}\right)_{1}\label{functions} 
\end{equation}
vanishes nowhere. Thus, assuming that $M$ is 2-nondegenerate, we have
\begin{equation}
\hbox{$S\ne 0$ everywhere.}\label{snonzero}
\end{equation}

Next, consider the fiber bundle ${\mathcal P}_{\Gamma_{\hspace{-0.05cm}{}_\rho}}\to {\Gamma}_{\rho}$ arising from the reduction to absolute parallelisms achieved in \cite{IZ} for CR-hypersurfaces in the class ${\mathfrak C}_{2,1}$,  and let $\gamma_0$ be the section of ${\mathcal P}_{\Gamma_{\hspace{-0.05cm}{}_\rho}}$ given in suitable coordinates by \cite[formula (4.21)]{I2}. In \cite[formula (4.27)]{I2} we computed the restriction of the curvature coefficient $\Theta^{2}_{21}$ to $\gamma_0$. The condition $\Theta^{2}_{21}|_{\gamma_0}=0$ can be then written as the equation
\begin{equation}
\makebox[250pt]{$\begin{array}{l}
\displaystyle2{\sqrt{\rho_{11}}}\left[\rho_{12}\left(\frac{S_{1}}{\sqrt{\rho_{11}}S}\right)_{\hspace{-0.1cm}1}-\rho_{11}\left(\frac{S_{1}}{\sqrt{\rho_{11}}S}\right)_{\hspace{-0.1cm}2}\right]-\\
\vspace{-0.1cm}\\
\displaystyle\hspace{0.4cm}2{\sqrt{\rho_{11}}}\left[\rho_{12}\left(\frac{\rho_{111}}{\sqrt{\rho_{11}^3}}\right)_{\hspace{-0.1cm}1}-\rho_{11}\left(\frac{\rho_{111}}{\sqrt{\rho_{11}^3}}\right)_{\hspace{-0.1cm}2}\right]-{11S_{1}}\,\rho_{11}-{S\,\rho_{111}}=0
\end{array}$}\label{veryfinalthetav}
\end{equation}
(cf.~\cite[formula (4.28)]{I2}). Further, in \cite[formula (4.46)]{I2} we found the expression for the restriction of $\Theta^{2}_{10}$ to $\gamma_0$ in which some of the terms were computed under the simplifying assumption that equation (\ref{veryfinalthetav}) holds. If we denote the quantity resulting from this calculation by ${\mathbf \Theta}^{2}_{10}$, then one observes that the condition ${\mathbf \Theta}^{2}_{10}=0$ can be written as the equation
\begin{equation}
9\rho^{{\rm(V)}}\rho_{11}^{2}-45\rho^{{\rm(IV)}}\rho_{111}\rho_{11}+40\rho_{111}^{3}=0, 
\label{veryfinalthetasss}
\end{equation}
where $\rho^{{\rm(IV)}}:=\partial^{\,4}\rho/\partial\,t_1^4$, $\rho^{{\rm(V)}}:=\partial^{\,5}\rho/\partial\,t_1^5$ (cf.~\cite[formula (4.47)]{I2}). Notice that, remarkably, (\ref{veryfinalthetasss}) is the Monge equation with respect to the first variable.

We are now ready to state and prove the detailed variant of Theorem \ref{main}:

\begin{theorem}\label{maindetailed}
Let $\rho$ be a smooth function satisfying {\rm (\ref{initial})--(\ref{rho11}) and (\ref{snonzero})}, where $S$ is defined in {\rm (\ref{functions})}. Then condition {\rm (\ref{veryfinalthetasss})} implies condition {\rm (\ref{veryfinalthetav})}. 
\end{theorem}

\begin{remark}\label{clarific}
We emphasize that although the quantity ${\mathbf \Theta}^2_{10}$ was computed partly under the assumption $\Theta^2_{21}|_{\gamma_0}=0$, the fact that the vanishing of ${\mathbf \Theta}^2_{10}$ implies the vanishing of $\Theta^2_{21}|_{\gamma_0}$ is not at all obvious and is actually  quite surprising.
\end{remark}

\begin{proof} We start by recalling classical facts concerning solutions of the homogeneous Monge-Amp\`ere equation (\ref{mongeampere}). For details the reader is referred to paper \cite{U}, which treats this equation in somewhat greater generality. 

Let us make the following change of variables near the origin:
\begin{equation}
\begin{array}{l}
v=\rho_1(t_1,t_2),\\
\vspace{-0.3cm}\\
w=t_2
\end{array}\label{changevar}
\end{equation}
and set
\begin{equation}
\begin{array}{l}
p(v,w):=\rho_2(t_1(v,w),w),\\
\vspace{-0.3cm}\\
q(v):=t_1(v,0).
\end{array}\label{condsfg}
\end{equation}
Equation (\ref{mongeampere}) immediately implies that $p$ is independent of $w$, so we write $p$ as a function of the variable $v$ alone. Furthermore, we have
\begin{equation}
q'(v)=\frac{1}{\rho_{11}(t_1(v,0),0)}.\label{gprime}
\end{equation}
Clearly, (\ref{initial}), (\ref{rho11}), (\ref{changevar}), (\ref{condsfg}), (\ref{gprime}) yield
\begin{equation}
p(0)=0,\quad q(0)=0,\quad \hbox{$q'>0$ everywhere.}\label{initialconds}
\end{equation}

In terms of $p$ and $q$, the inverse of (\ref{changevar}) is written as
\begin{equation}
\begin{array}{l}
t_1=q(v)-w\,p'(v),\\
\vspace{-0.3cm}\\
t_2=w,
\end{array}\label{inverttted}
\end{equation}
and the solution $\rho$ in the variables $v,w$ is given by
\begin{equation}
\rho(t_1(v,w),w)=vq(v)-\int_{0}^vq(\tau)d\tau+w(p(v)-vp'(v)).\label{solsparam}
\end{equation}
In particular, we see that the general smooth solution of the homogeneous Monge-Amp\`ere equation (\ref{mongeampere}) satisfying conditions (\ref{initial}), (\ref{rho11}) is parametrized by a pair of arbitrary smooth functions satisfying (\ref{initialconds}). 

We will now rewrite equation (\ref{veryfinalthetav}) in the variables $v$, $w$ introduced in (\ref{changevar}). First of all, from (\ref{changevar}), (\ref{inverttted}) we compute
\begin{equation}
\begin{array}{l}
\displaystyle\rho_{11}(t_1(v,w),w)=\displaystyle\frac{1}{q'-w\,p''},\\
\vspace{-0.3cm}\\
\displaystyle\rho_{12}(t_1(v,w),w)=\displaystyle\frac{p'}{q'-w\,p''},\\
\vspace{-0.3cm}\\
\displaystyle \rho_{111}(t_1(v,w),w)=-\frac{q''-w\,p'''}{(q'-w\,p'')^3}.
\end{array}\label{secondpartials}
\end{equation}
Next, from formulas (\ref{functions}), (\ref{inverttted}) and the first two identities in (\ref{secondpartials}) we obtain
\begin{equation}
\begin{array}{l}
\displaystyle S(t_1(v,w),w)=\frac{p''}{q'-w\,p''},\\
\vspace{-0.1cm}\\
\displaystyle S_1(t_1(v,w),w)=\frac{p'''q'-p''q''}{(q'-w\,p'')^3}.\\
\end{array}\label{ids444}
\end{equation}

Now, plugging the expressions from (\ref{secondpartials}), (\ref{ids444}) into (\ref{veryfinalthetav}), we see that the latter simplifies to the equation
\begin{equation}
p'''q'-p''q''= 0,\label{vanish1}
\end{equation}
that is, to the condition $S_1=0$. Since $S$ vanishes nowhere, the first identity in (\ref{ids444}) implies that $p''$ does not vanish either (this condition characterizes 2-nondegeneracy). Then, dividing (\ref{vanish1}) by $(p'')^2$, one obtains 
\begin{equation}
q'/p''=\hbox{const}.\label{firstcur}
\end{equation}
Thus, we see that after passing to the variables $v$, $w$ the complicated equation (\ref{veryfinalthetav}) turns into the simple condition (\ref{firstcur}). 

Further, we will rewrite equation (\ref{veryfinalthetasss}) in the variables $v$, $w$. From (\ref{inverttted}) and (\ref{secondpartials}) one computes 
\begin{equation}
\hspace{0.8cm}\makebox[250pt]{$\begin{array}{l}
\displaystyle \rho^{{\rm(IV)}}(t_1(v,w),w)=-\frac{1}{(q'-w\,p'')^5}\Bigl[(q'''-w\,p^{{\rm(IV)}})(q'-w\,p'')-\\
\vspace{-0.6cm}\\
\displaystyle\hspace{9cm}3(q''-w\,p''')^2\Bigr],\\
\vspace{-0.1cm}\\
\displaystyle \rho^{{\rm(V)}}(t_1(v,w),w)=-\frac{1}{(q'-w\,p'')^7}\Bigl[\Bigl((q^{{\rm(IV)}}-w\,p^{{\rm(V)}})(q'-w\,p'')-\\
\vspace{-0.4cm}\\
\displaystyle\hspace{1cm}5(q''-w\,p''')(q'''-w\,p^{{\rm(IV)}})\Bigr)(q'-w\,p'')-\\
\vspace{-0.4cm}\\
\displaystyle\hspace{1cm}5\Bigl((q'''-w\,p^{{\rm(IV)}})(q'-w\,p'')-3(q''-w\,p''')^2\Bigr)(q''-w\,p''')\Bigr].
\end{array}$}\label{rho4and5}
\end{equation}
Plugging expressions from (\ref{secondpartials}), (\ref{rho4and5}) into (\ref{veryfinalthetasss}) and collecting coefficients at $w^k$ for $k=0,1,2,3$ in the resulting formula, we see that (\ref{veryfinalthetasss}) is equivalent to the following system of four ordinary differential equations:
\begin{equation}
\left\{
\begin{array}{l}
9p^{{\rm(V)}}(p'')^2-45p^{{\rm(IV)}}p'''p''+40(p''')^3=0,\\
\vspace{-0.1cm}\\
6p^{{\rm(V)}}p''q'+3(p'')^2q^{{\rm(IV)}}-15(p^{{\rm(IV)}}p'''q'+p^{{\rm(IV)}}p''q''+p'''p''q''')+\\
\vspace{-0.3cm}\\
\hspace{8cm}40(p''')^2q''=0,\\
\vspace{-0.3cm}\\
3p^{{\rm(V)}}(q')^2+6p''q^{{\rm(IV)}}q'-15(p^{{\rm(IV)}}q''q'+p'''q'''q'+p''q'''q'')+\\
\vspace{-0.3cm}\\
\hspace{8cm}40p'''(q'')^2=0,\\
\vspace{-0.3cm}\\
9q^{{\rm(IV)}}(q')^2-45q'''q''q'+40(q'')^3=0.\\
\end{array}
\right.\label{final1}
\end{equation}
Thus, in order to prove the theorem, we need to show that system (\ref{final1}) implies condition (\ref{firstcur}). 

Notice that the first equation in (\ref{final1}) is the Monge equation and that the last one yields the Monge equation for any primitive of the function $q$. Also observe that all the equations in (\ref{final1}) reduce to the first one if condition (\ref{firstcur}) is satisfied.  

Recall now that the Monge equation describes planar conics and can be solved explicitly. Indeed, assuming that $p''>0$ we calculate
$$
\begin{array}{l}
\displaystyle\frac{1}{(p'')^{11/3}}\Bigl(9p^{{\rm(V)}}(p'')^2-45p^{{\rm(IV)}}p'''p''+40(p''')^3\Bigr)=\\
\vspace{-0.1cm}\\
\displaystyle\left(\frac{9p^{{\rm(IV)}}}{(p'')^{5/3}}-\frac{15(p''')^2}{(p'')^{8/3}}\right)'=9\left(\frac{p'''}{(p'')^{5/3}}\right)''=-\frac{27}{2}\left((p'')^{-2/3}\right)'''.
\end{array}
$$
Similarly, for $p''<0$ we have
$$
\begin{array}{l}
\displaystyle\frac{1}{(-p'')^{11/3}}\Bigl(9p^{{\rm(V)}}(p'')^2-45p^{{\rm(IV)}}p'''p''+40(p''')^3\Bigr)=\frac{27}{2}\left((-p'')^{-2/3}\right)'''.
\end{array}
$$
Thus, the first equation in (\ref{final1}) yields
\begin{equation}
p''=\pm P^{-3/2},\label{pprpr}
\end{equation}
where $P$ is a polynomial with $\deg P\le 2$ and $P(0)>0$. Similarly, taking into account (\ref{initialconds}), from the last equation in (\ref{final1}) we see
\begin{equation}
q'= Q^{-3/2},\label{qpr}
\end{equation}
where $Q$ is a polynomial with $\deg Q\le 2$ and $Q(0)>0$.

Next, from (\ref{pprpr}), (\ref{qpr}) we calculate
\begin{equation}
\begin{array}{l}
\displaystyle p'''=\mp\frac{3}{2}P^{-5/2}P',\\
\vspace{-0.1cm}\\
\displaystyle p^{{\rm(IV)}}=\pm\frac{15}{4}P^{-7/2}(P')^2\mp\frac{3}{2}P^{-5/2}P'',\\
\vspace{-0.1cm}\\
\displaystyle p^{{\rm(V)}}=\mp\frac{105}{8}P^{-9/2}(P')^3\pm\frac{45}{4}P^{-7/2}P''P',\\
\vspace{-0.1cm}\\
\displaystyle q''=-\frac{3}{2}Q^{-5/2}Q',\\
\vspace{-0.1cm}\\
\displaystyle q'''=\frac{15}{4}Q^{-7/2}(Q')^2-\frac{3}{2}Q^{-5/2}Q'',\\
\vspace{-0.1cm}\\
\displaystyle q^{{\rm(IV)}}=-\frac{105}{8}Q^{-9/2}(Q')^3+\frac{45}{4}Q^{-7/2}Q''Q'.
\end{array}\label{derivs}
\end{equation}
Plugging (\ref{pprpr}), (\ref{qpr}), (\ref{derivs}) into the second and third equations in (\ref{final1}) and simplifying the resulting expressions, we obtain, respectively,
\begin{equation}
\begin{array}{l}
7P^3(Q')^3-6P^3Q''Q'Q-(P')^3Q^3+9(P')^2PQ'Q^2-6P''P'PQ^3-\\
\vspace{-0.3cm}\\
\hspace{3.3cm}15P'P^2(Q')^2Q+6P'P^2Q''Q^2+6P''P^2Q'Q^2=0,\\
\vspace{-0.1cm}\\
7(P')^3Q^3-6P''P'PQ^3-P^3(Q')^3+9P'P^2(Q')^2Q-6P^3Q''Q'Q-\\
\vspace{-0.3cm}\\
\hspace{3.3cm}15(P')^2PQ'Q^2+6P''P^2Q'Q^2+6P'P^2Q''Q^2=0.
\end{array}\label{PQ}
\end{equation}
Subtracting the second identity in (\ref{PQ}) from the first one, we arrive at
$$
8(PQ'-P'Q)^3=0.
$$
It then follows that $Q=\hbox{const}\, P$, and therefore condition (\ref{firstcur}) holds as required. The proof is complete. \end{proof}

\section{A class of nonflat tubes}\label{secother}
\setcounter{equation}{0}

In this section we investigate the question of whether for a hypersurface of the form (\ref{basiceq}) that is uniformly Levi degenerate of rank 1 and 2-nondegenerate the condition $\Theta^{2}_{21}|_{\gamma_0}=0$ yields CR-flatness. Equivalently, 
\vspace{-0.5cm}\\

$$
\begin{array}{l}
\hspace{0.2cm}\hbox{$(**)$ given a smooth function $\rho$ satisfying {\rm (\ref{initial})--(\ref{rho11}) and (\ref{snonzero})},}\\
\vspace{-0.3cm}\hspace{0.95cm}\hbox{does condition (\ref{veryfinalthetav}) imply condition (\ref{veryfinalthetasss})?}\\
\end{array}
$$
\vspace{-0.4cm}\\

\noindent If one looks at (\ref{veryfinalthetav}) and (\ref{veryfinalthetasss}) in their original complicated PDE form, this question may appear to be hard. Luckily, after passing to the variables $v$, $w$ as in (\ref{changevar}), both equations simplify: (\ref{veryfinalthetav}) becomes condition (\ref{firstcur}), whereas (\ref{veryfinalthetasss}) turns into the system of four ordinary differential equations (\ref{final1}). In fact, as we remarked in the proof of Theorem \ref{maindetailed} in the preceding section, (\ref{firstcur}) forces all the equations in (\ref{final1}) to be identical to the Monge equation on the function $p$, thus, assuming that (\ref{firstcur}) holds, (\ref{veryfinalthetasss}) actually reduces to a single ODE. Therefore, for any $p$ that is {\it not}\, a solution of the Monge equation and such that 
\begin{equation}
p(0)=0,\quad \hbox{$p''\ne 0$ everywhere,}\label{condss1}
\end{equation}
and for
\begin{equation} 
\hbox{$q:=C (p'-p'(0))$, with $Cp''>0$},\label{condss2}
\end{equation}
formula (\ref{solsparam}) provides a counterexample to question $(**)$ (see, e.g., Example \ref{ex} below). 

In fact, for $p$ having properties (\ref{condss1}) and $q$ chosen as in (\ref{condss2}), formula (\ref{solsparam}) significantly simplifies:

\begin{proposition}\label{main1}
Let $\rho$ be a smooth function satisfying {\rm (\ref{initial})--(\ref{rho11}), (\ref{snonzero})} and {\rm (\ref{veryfinalthetav})}. Then formula {\rm (\ref{solsparam})} becomes
\begin{equation}
\rho(t_1(v,w),w)=(w-C)(p(v)-vp'(v)).\label{solsparam1}
\end{equation}
\end{proposition}

\begin{proof} Substituting (\ref{condss2}) into (\ref{solsparam}) we calculate
$$
\begin{array}{l}
\displaystyle\rho(t_1(v,w),w)=vC (p'(v)-p'(0))-\int_0^vC (p'(\tau)-p'(0))d\tau+w(p(v)-vp'(v))=\\
\vspace{-0.3cm}\\
vC (p'(v)-p'(0))-C(p(v)-p'(0)v)+w(p(v)-vp'(v))=(w-C)(p(v)-vp'(v))
\end{array}
$$
as required.\end{proof}

Thus, all hypersurfaces of the form (\ref{basiceq}) that are 2-nondegenerate, uniformly Levi degenerate of rank 1, and for which $\Theta^{2}_{21}|_{\gamma_0}=0$, are described by formula (\ref{solsparam1}). This is an interesting class of not necessarily CR-flat tubes, and it is quite useful to have an explicit characterization for it. Notice, however, that although formula (\ref{solsparam1}) is very simple, it is written in the variables $v$, $w$, whereas the expression for $\rho$ in the original variables $t_1$, $t_2$ (which is what we are really interested in) may turn out to be more complicated. This expression was found in \cite[Lemma 4.1]{I2}, and, as the argument is quite short, we repeat it here for the sake of the completeness of our exposition. 

Let $\zeta$ be the inverse of the function $p'(0)-p'$ near the origin. Define
\begin{equation}
\chi(\tau):=\frac{1}{\tau}\int_{0}^{\tau}\zeta(\sigma)d\sigma.\label{functchi}
\end{equation}
Clearly, $\chi$ is smooth near 0 and satisfies
\begin{equation}
\chi(0)=0,\quad \chi'(0)=-\frac{1}{2p''(0)}.\label{chiconds}
\end{equation} 
Now set
\begin{equation}
\tilde\rho(t_1,t_2):=(t_1+ p'(0) t_2)\chi\left(\frac{t_1+ p'(0) t_2}{t_2-C}\right).\label{deftilderho}
\end{equation}

\begin{proposition}\label{firstcondrho} \it One has $\rho=\tilde\rho$.
\end{proposition}

\begin{proof}
From (\ref{deftilderho}) we compute:
\begin{equation}
\hspace{0.4cm}\makebox[250pt]{$\begin{array}{l}
\displaystyle\tilde\rho_1=\chi\left(\frac{t_1+ p'(0) t_2}{t_2-C}\right)+\frac{t_1+ p'(0) t_2}{t_2-C}\chi'\left(\frac{t_1+ p'(0) t_2}{t_2-C}\right)=\zeta\left(\frac{t_1+ p'(0) t_2}{t_2-C}\right),\\
\vspace{-0.1cm}\\
\displaystyle\tilde\rho_2=p'(0)\,\chi\left(\frac{t_1+ p'(0) t_2}{t_2-C}\right)+\\
\vspace{-0.4cm}\\
\displaystyle\hspace{3cm}\frac{t_1+ p'(0) t_2}{t_2-C}\left(p'(0)-\frac{t_1+p'(0)t_2}{t_2-C}\right)\chi'\left(\frac{t_1+ p'(0) t_2}{t_2-C}\right),\\
\vspace{-0.1cm}\\
\displaystyle\tilde\rho_{11}=\frac{2}{t_2-C}\chi'\left(\frac{t_1+ p'(0) t_2}{t_2-C}\right)+\frac{t_1+ p'(0) t_2}{(t_2-C)^2}\chi''\left(\frac{t_1+ p'(0) t_2}{t_2-C}\right).
\end{array}$}\label{tilderhoderiv}
\end{equation}
Formulas (\ref{chiconds}), (\ref{deftilderho}), (\ref{tilderhoderiv}) imply
$$
\tilde\rho(0)=0,\quad\tilde\rho_1(0)=0,\quad \tilde\rho_2(0)=0,\quad \tilde\rho_{11}>0.
$$
Also, it is easy to observe that $\tilde\rho$ satisfies the Monge-Amp\`ere equation (\ref{mongeampere}). Hence, $\tilde\rho$ is fully determined by a pair of functions $\tilde p$, $\tilde q$ as in formulas (\ref{inverttted}),  (\ref{solsparam}). These functions satisfy
$$
\tilde p(0)=0,\quad \tilde q(0)=0,\quad \hbox{$\tilde q\,'>0$ everywhere}
$$
(cf.~conditions (\ref{initialconds})).

Let us make a change of coordinates near the origin analogous to (\ref{changevar}):
\begin{equation}
\begin{array}{l}
\tilde v=\tilde\rho_1(t_1,t_2),\\
\vspace{-0.3cm}\\
\tilde w=t_2.
\end{array}\label{changevar1}
\end{equation}
Then by the first identity in (\ref{tilderhoderiv}) we have
$$
\tilde v=(p'(0)-p')^{-1}\left(\frac{t_1+ p'(0) t_2}{t_2-C}\right)
$$
and therefore, taking into account (\ref{condss2}), we see that (\ref{changevar1}) is inverted as
$$
\begin{array}{l}
t_1=C(p'(\tilde v)-p'(0))-\tilde wp'(\tilde v)=q(\tilde v)-\tilde w p'(\tilde v),\\
\vspace{-0.3cm}\\
t_2=\tilde w.
\end{array}
$$
On the other hand, as in (\ref{inverttted}) we have
$$
t_1=\tilde q(\tilde v)-\tilde w \tilde p\,'(\tilde v).
$$
Hence, it follows that $\tilde q=q$ and, since $\tilde p(0)=p(0)=0$, one also has $\tilde p=p$. Therefore, $\tilde\rho=\rho$,  and the proof is complete. \end{proof}

We will now demonstrate how Propositions \ref{main1}, \ref{firstcondrho} work for a particular example.

\begin{example}\label{ex}\rm 

Let $p(v)=e^v-1$. Clearly, conditions (\ref{condss1}) hold  for this choice of $p$. Then by formula (\ref{solsparam1}) we compute
\begin{equation}
\rho(t_1(v,w),w)=(w-C)((1-v)e^v-1),\label{formrho}
\end{equation}
where $C>0$. To rewrite $\rho$ in the variables $t_1$, $t_2$, we can either directly invert formula (\ref{inverttted}) or use Proposition \ref{firstcondrho}. To invert formula (\ref{inverttted}), we notice that for our choice of $p$ and $q$ it becomes
$$
\begin{array}{l}
t_1=(C-w)e^v-C,\\
\vspace{-0.3cm}\\
t_2=w.
\end{array}
$$
We then obtain
\begin{equation}
\begin{array}{l}
\displaystyle v=\log\left(\frac{t_1+C}{C-t_2}\right),\\
\vspace{-0.3cm}\\
w=t_2,
\end{array}\label{inverse}
\end{equation}
and plugging (\ref{inverse}) into (\ref{formrho}) yields
\begin{equation}
\rho(t_1,t_2)=(t_1+C)\log\left(\frac{t_1+C}{C-t_2}\right)-(t_1+t_2).\label{formmmrho}
\end{equation}

Rather than inverting formula (\ref{inverttted}), let us now utilize Proposition \ref{firstcondrho} in order to determine $\rho(t_1,t_2)$. We have $\zeta(\sigma)=\log(1-\sigma)$, and therefore by (\ref{functchi}) we see
$$
\chi(\tau)=\frac{\tau-1}{\tau}\log(1-\tau)-1.
$$
Then after a short calculation formula (\ref{deftilderho}) leads to expression (\ref{formmmrho}) as well.

Note that, since the function $p(v)$ in this example does not satisfy the Monge equation, the corresponding tube hypersurface $\Gamma_{\rho}$ defined by (\ref{basiceq}) is not CR-flat, or, equivalently, the quantity $\Theta^{2}_{10}|_{\gamma_0}={\mathbf \Theta}^{2}_{10}$ does not identically vanish.
\end{example}

\begin{remark}\label{pocfunctions}
For any real hypersurface $M$ in $\CC^3$ in the class ${\mathfrak C}_{2,1}$, paper \cite{Poc} introduces a pair of expressions, called $J$ and $W$, in terms of a local defining function that vanish simultaneously on $M$ if and only if $M$ is locally CR-equivalent to $M_0$. The expressions are rather complicated, but in the tube case it is not very hard to see that the condition $W=0$ is identical to equation (\ref{veryfinalthetav}) (i.e., to the vanishing of $\Theta^{2}_{21}|_{\gamma_0}$) and the condition $J=0$ calculated under the assumption $W=0$ to equation (\ref{veryfinalthetasss}) (i.e., to the vanishing of $\Theta^{2}_{10}|_{\gamma_0}$ calculated in part under the assumption $\Theta^{2}_{21}|_{\gamma_0}=0$ as in \cite{I2}). It would be interesting to see whether an analogue of Theorem \ref{main} holds for $J$ and $W$ in place of $\Theta^{2}_{10}|_{\gamma_0}$, $\Theta^{2}_{21}|_{\gamma_0}$, respectively, if the hypersurface is no longer assumed to be tube, i.e., whether it is possible to find a reasonable single condition characterizing CR-flatness for the entire class ${\mathfrak C}_{2,1}$.
\end{remark}

\end{document}